\documentclass{amsart}

\usepackage{amssymb,color}
\usepackage{amsfonts}
\usepackage{amsmath}
\usepackage{euscript}
\usepackage{enumerate}
\usepackage{pdfsync}
\newtheorem{theorem}{Theorem}[section]
\newtheorem{lemma}[theorem]{Lemma}

\newtheorem{prop}[theorem]{Proposition}
\newtheorem{cor}[theorem]{Corollary}

\newtheorem*{Theorem1'}{Theorem 1'}

\theoremstyle{definition}

\theoremstyle{remark}

\def\Z{{\bf Z}}

\def\N{{\bf N}}

\begin{document}

\title[A power-commutator presentation]
{A power-commutator presentation for the group of truncated polynomials under substitution}

%\author{Agust\'{i}n D'Alessandro}
%\address{Department of Mathematics and Statistics, University of Regina, Canada}
%\email{dalessandro.ag@gmail.com}

\author{F. Szechtman}
\address{Department of Mathematics and Statistics, University of Regina, Canada}
\email{fernando.szechtman@gmail.com}
\thanks{The author was supported in part by NSERC discovery grant 2020-04062}


%\date{January 1, 2001 and, in revised form, June 22, 2001.}



\begin{abstract} Given a prime $p$, the Nottingham group $G(p)$ is the group of formal power series
$x+a_2x^2+a_3x^3+\cdots$, $a_i\in\Z/p\Z$, under substitution. For $n\geq 1$, we write $G_n(p)$
for the quotient of $G(p)$ by its normal subgroup $K_n=\{x+a_{n+1}x^{n+1}+a_{n+2}x^{n+2}+\cdots\,|\, a_i\in\Z/p\Z\}$,
namely the group of truncated polynomials $x+a_2x^2+\cdots+a_nx^n$, $a_i\in\Z/p\Z$, under substitution, which
is a $p$-group of order $p^{n-1}$ generated by $x+x^2,\dots,x+x^n$. In this paper, we determine the
power-commutator presentation of $G_{15}(p)$ relative to $x+x^2,\dots,x+x^{15}$. This, in turn, automatically gives
the power-commutator presentation of $G_{n}(p)$ relative to $x+x^2,\dots,x+x^{n}$ for any $1\leq n<15$,
as well as the initial segments of the power and commutator relations in $G_{n}(p)$ 
corresponding to $x+x^2,\dots,x+x^{15}$ for any $n>15$. 
\end{abstract}

\maketitle

\section{Introduction}

Let $G(R)$ be the group of power series $x+a_2x^2+a_3x^3+\cdots\in R[[x]]$ under substitution,
where $R$ is a commutative ring  with $1\neq 0$, so that
$f*g=f(g)$ for $f,g\in G(R)$. For $n\geq 1$, we set
$$
K_n=\{x+a_{n+1}x^{n+1}+a_{n+2}x^{n+2}+\cdots\,|\, a_i\in R\},
$$
a normal subgroup of $G(R)$, with quotient group $G_n(R)=\{x+a_{2}x^2+\cdots+a_n x^n\,|\,  a_i\in R\}$, the group of truncated polynomials over~$R$ of degree 
$\leq n$ under substitution, denoted by $G_n(m)$ when $R=\Z/m\Z$, $m>1$, which is a group of order $m^{n-1}$.

The first systematic study of $G$ was made
by Jennings \cite{J} in 1954, followed three decades later by Johnson \cite{Jo} and York \cite{Y,Y2}.
Ever since then, $G$ has received and continues to attract considerable attention
due to its remarkable properties and the important role it plays in the theory of pro-$p$-groups when
$R$ is a finite field, in which case $G$ is usually known as the Nottingham group. 
We refer to the survey articles by Camina \cite{C} and Babenko \cite{B}, and references therein, for a wealth of 
properties enjoyed by $G$.

In this paper, we determine the
power-commutator presentation of $G_{15}(p)$, $p$ a prime, relative to $x+x^2,\dots,x+x^{15}$. This, in turn, automatically gives
the power-commutator presentation of $G_{n}(p)$ relative to $x+x^2,\dots,x+x^{n}$ for any $1\leq n<15$,
by simply eliminating all generators~$x+x^i$, $n<i\leq 15$, and all relations where they appear; we also obtain
as a by-product all initial segments of the power and commutator relations of $G_n(p)$, $n>15$, corresponding to $x+x^2,\dots,x+x^{15}$,
which must be prolonged by appending the required exponents of $x+x^{16},\dots,x+x^n$ on the right hand side, and completed
to the full power-commutator presentation of $G_{n}(p)$ by calculating all power and commutator relations involving 
$x+x^{16},\dots,x+x^n$ on the left hand side (in general, this second task is trivial when passing from the power-commutator
presentation of $G_{n}(p)$ to that of $G_{n+1}(p)$ as $x+x^{n+1}$ is central of order $p$; it is the foregoing prolongation that requires effort).

To derive the commutator relations in $G_{15}(p)$, we develop in Section \ref{cf} a general commutator formula applicable
to $x+x^2,x+x^{3},\dots\in G(R)$, which is substantial sharpening of Jenning's commutator formula \cite[Section 2]{J},
valid for arbitrary elements of $G(R)$. Our formula explicitly gives the
full commutator $[x+x^i,x+x^j]$, $2\leq j<i\leq n$, in $G_{n}(R)$ for the vast majority of choices of $i$ and $j$
when $n$ is as large as $100$ (see Section \ref{cf} for details) and many of the initial terms in the remaining cases.
When $n=15$, only a few commutators are not fully obtained by means of our commutator formula, and
we find the missing terms of these commutators by solving the equation $(x+x^j)*(x+x^i)=[x+x^i,x+x^j]*(x+x^i)*(x+x^j)$
one coefficient of $[x+x^i,x+x^j]$ at a time (the initial coefficients are given by Proposition \ref{col}). The commutator
requiring the greatest effort is $[x+x^3,x+x^2]$. The determination of $[x+x^i,x+x^j]$ is carried out in polynomial
form. We must then write these polynomials as a product of generators $x+x^k$, $k>i$, in increasing order
and raised to suitable exponents. This factorization is unique when $R=\Z$ and automatically gives a factorization when $R$ is arbitrary.
The tools required to deal with this step are developed in Section \ref{nt} and the application of these tools to $G_{15}(R)$
is carried out in Section~\ref{nt2}. In our last step, Section \ref{sa} deals with case when $R$ has characteristic $p$ and determines
what changes are required to the previous factorization to produce another one where all exponents are between 0 and $p-1$.

The power relations in $G_{15}(p)$ are initially approached by means of a non-trivial power formula
recently developed in \cite{DS}. The most challenging cases cannot be fully solved in this way,
and require the use of York's \cite{Y} matrix representation of $G(R)$. This part of our work
can be found in Section \ref{pr}, and the final power-commutator presentation of $G_{15}(p)$
is described in Theorem \ref{pcp}.

The above program was carried out by hand calculations and later verified using GAP. 

Appealing to the canonical projection $G_n(p)\to G_{n-1}(p)$, $n>1$,
to derive information about $\mathrm{Aut}(G_n(p))$ from that of $\mathrm{Aut}(G_{n-1}(p))$,
we will use the power-commutator presentation of $G_{15}(p)$ obtained here to study
the automorphism group of $G_{n}(p)$ for arbitrary $n$. The finite groups $G_{n}(p)$ have considerable more symmetry
than the Nottingham group $G(p)$, $p\geq 5$, studied by Klopsch~\cite{K}, all whose automorphisms are induced
by an inner automorphism $(\Z/p\Z)^\times\ltimes G(p)$ (where $\Z_p^\times$ is viewed
as the group of all $ax$, $a\in \Z_p^\times$, acting on $G(p)$ by conjugation).

York \cite[Chapter 7]{Y2} worked out power-commutator presentations of $G_5(p)$ and $G_6(p)$,
commenting at the end of the chapter that the general case of $G_n(p)$, $n\geq 7$, rapidly becomes 
difficult to handle.

Following York~\cite{Y}, we write $f^{(m)}$ for the $m$th iterate of 
$f\in G(R)$ to distinguish it from its $m$th power $f^m$ in $R[[x]]$.

We write $\pi_n:G(R)\to G_n(R)$ for the canonical projection,
setting $H_r=\pi_n(K_r)$, $r\geq 1$; this is trivial
if $n\leq r$, and the subgroup $H_r=\{x+a_{r+1}x^{r+1}+\cdots+a_n x^n\,|\,  a_i\in R\}$ of $G_n(R)$ if $r<n$.

\section{Power-commutator presentations} 

Let $G$ be a group of order $p^m$, where $p$ is a prime number and $m$ is a natural number.
Then $G$ has a central series with cyclic factors of order $p$, so there exist generators
$x_1,\dots,x_m$ of $G$ such that for each $1\leq i<m$, we have
\begin{equation}
\label{t1}
x_i^p=x_{i+1}^{a(i,i+1)}\cdots x_{m}^{a(i,m)}
\end{equation}
where all exponents are between $0$ and $p-1$, with the understanding that $x_m^p=1$, and 
given any $1\leq j<i\leq m$, we have
\begin{equation}
\label{t2}
[x_i,x_j]=x_{i+1}^{b(j,i,i+1)}\cdots x_{m}^{b(j,i,m)},
\end{equation}
where all exponents are between $0$ and $p-1$, with the understanding that $[g_m,g_i]=1$ for all $1\leq i<m$.

Any group $H=\langle h_1,\dots,h_m\rangle$ satisfying (\ref{t1}) and (\ref{t2}) (where
each $x_i$ is to be replaced by $h_i$) is easily seen to have order at most $p^m$. It follows that
the relations (\ref{t1}) and (\ref{t2}) of $G$ are defining relations, yielding the
power-commutator presentation (PCP) of $G$ relative to $x_1,\dots,x_m$. Any element of $\langle x_1,\dots,x_i\rangle$,  $1\leq i\leq m$,
can be written in the form $x_1^{c_1},\dots,x_i^{c_i}$ for unique integers $0\leq c_j<p$.

\section{Background}

It was shown by Jennings \cite[Theorem 2.5]{J} that if $R$ has prime characteristic $p$, then $K_r^{(p)}\subseteq K_{pr}$.
The following sharpening of this result can be found in Camina's survey article \cite[Section 4]{C}.

\begin{theorem}\label{kr} If $R$ has prime characteristic $p$ and $t$ is the remainder of dividing $r$ by $p$,  then
$K_r^{(p)}\subseteq K_{pr+t}$.
\end{theorem}

Camina states Theorem \ref{kr} when $R$ is a finite field. A proof of the general case can be found in \cite[Theorem 3.5]{DS}.
The following commutator formula is also due to Jennings \cite[Section 2]{J}. For an alternative proof, see \cite[Proposition 3.2]{DS}.

\begin{prop}\label{com1} If $r,s>1$ and 
$$
f=x+f_rx^r+\cdots\in K_{r-1},\; g=x+g_sx^s+\cdots \in K_{s-1},
$$
then
$$
[f,g]=x+(r-s)f_r g_s x^{r+s-1}+\cdots\in K_{(r-1)+(s-1)}.
$$
\end{prop}

\begin{cor}\label{kbr} For $n,m\in\N$, we have
$
[K_n,K_m]\subseteq K_{n+m}.
$
In  particular,
$
[G(R), K_m]\subseteq K_{m+1},
$
so
$
G(R)=K_1\supset K_2\supset\cdots
$
is a central series of $G(R)$.
\end{cor}

\begin{lemma}\label{nat} For any $r\geq 1$, the map $\phi:K_r\to R^+$,
defined by $$\phi(x+a_{r+1}x^{r+1}+a_{r+2}x^{r+2}+\cdots)=a_{r+1},$$
is a group epimorphism with kernel $K_{r+1}$. Thus, $(x+a x^{r+1})*K_{r+1}=(x+b x^{r+1})*K_{r+1}$
only when $a=b$; moreover,
$$
x+a_{r+1}x^{r+1}+a_{r+2}x^{r+2}+\cdots=(x+a_{r+1}x^{r+1})*(x+b_{r+2}x^{r+2}+b_{r+3}x^{r+3}+\cdots),
$$
for a unique $x+b_{r+2}x^{r+2}+b_{r+3}x^{r+3}+\cdots\in K_{r+1}$.
\end{lemma}

\begin{proof} The first statement is easy to see, and the rest follows from it.
\end{proof}

\begin{cor}\label{gen3} For $n>1$, 
$$
G_n(m)=H_1\supset H_2\supset\cdots\supset H_{n-1}\supset 1
$$
is a central series of $G_n(m)$ with cyclic factors of order $m$, and
$$
H_i=\langle x+x^{i+1},\dots,x+x^n\rangle
$$
for all $1\leq i<n$. 
\end{cor}

\begin{proof} This is a consequence of Corollary \ref{kbr} and Lemma \ref{nat}.
\end{proof}

We henceforth set $g_2=x+x^2,g_3=x+x^3,\dots$ (context will make it clear if
we view these as elements of $G(R)$ or $G_n(R)$) 
and proceed to find the PCP of $G_{15}(p)$ relative to $g_2,\dots,g_{15}$
when $p$ is a prime number that is also fixed from now on.

\section{A commutator formula}\label{cf}

%Given $\alpha\in\R$ and $n\geq 0$, recall that $\binom{\alpha}{n}=\alpha(\alpha-1)\cdots(\alpha-(n-1))/n!$,  where
%$\binom{\alpha}{0}=1$.

\begin{prop}\label{col} Let $2\leq j<i$. Then
$$
[g_i,g_j]\equiv x+a x^{j+i-1}+b x^{j+2i-2}+c x^{i+2j-2}+d x^{i+3j-3}+e x^{j+3i-3}+h x^{2i+2j-3}\mod x^{i+4j-4},
$$
where
$$
a=i-j,\; c=\binom{i}{2}-j(i-j),
$$
and if $2j\neq i+1$, then
$$
b=-\binom{j}{2}-i(i-j),\; d=\binom{i}{3}-j\binom{i}{2}+j^2(i-j),
$$
$$
e=\binom{j}{3}+i\binom{j}{2}+i^2(i-j),\; h=j\binom{j}{2}+j(i-j)(i-j+1)-i\binom{i}{2},
$$
while if $2j=i+1$ (which is equivalent to $j+2i-2=i+3j-3$), then $d=0$ and
$$
b=\binom{i}{3}-\binom{j}{2}+(j^2-i)(i-j)-j\binom{i}{2},\;
e=-\binom{j}{3}-i\binom{i}{3}+i\binom{j}{2}+i(i-j)(i-j^2)+ij\binom{i}{2},
$$
$$
h=-j\binom{i}{3}+j\binom{j}{2}+j(i-j)(i-j-j^2+1)+(j^2-i)\binom{i}{2}.
$$
Moreover, it is possible for any of $j+2i-2< 2i+2j-3<j+3i-3$ to be larger than or equal to $i+4j-4$, in which case the corresponding value of $b,h,e$ is completely arbitrary.
\end{prop}

\begin{proof} By  the binomial formula, we have
\begin{equation}\label{pup}
g_i*g_j=x+x^j+\sum_{k=0}^{i}\binom{i}{k}x^{i+kj-k},
\end{equation}
\begin{equation}\label{pup2}
g_j*g_i=x+x^i+\sum_{k=0}^{j}\binom{j}{k}x^{j+ki-k},
\end{equation}
and therefore
\begin{equation}\label{gigj}
g_i*g_j\equiv x+x^j+x^i+i x^{i+j-1}+\binom{i}{2}x^{i+2j-2}+\binom{i}{3}x^{i+3j-3}\mod \mod x^{i+4j-4},
\end{equation}
\begin{equation}\label{gjgi}
h_j*h_i\equiv x+x^i+x^j+j x^{j+i-1}+\binom{j}{2}x^{j+2i-2}+\binom{j}{3}x^{j+3i-3}\mod \mod x^{i+4j-4}.
\end{equation}
If $2j\neq i+1$, a careful calculation that makes use of (\ref{gigj}) and (\ref{gjgi})
reveals the existence of $a,b,c,d,e,h$ in $R$ such that
$$
g_j*g_i*(x+a x^{j+i-1}+b x^{j+2i-2}+c x^{i+2j-2}+d x^{i+3j-3}+e x^{j+3i-3}+h x^{2i+2j-3})\equiv g_i*g_j\mod x^{i+4j-4},
$$
and $a,b,c,d,e,h$ must satisfy
$$
a+j=i,\; b+ia+\binom{j}{2}=0,  c+ja=\binom{i}{2}, d+jc=\binom{i}{3}, e+ib+\binom{j}{3}=0, h+jb+ic+a(j+i-1)=0,
$$
with solutions as indicated. If $2j=i+1$ the argument is the same, except that the term $d x^{i+3j-3}$ must be removed
prior to the foregoing calculation, leading to
$$
a+j=i,\; c+ja=\binom{i}{2},\; b+ia+\binom{j}{2}+jc=\binom{i}{3},\; e+ib+\binom{j}{3}=0,\; h+jb+ic+a(j+i-1)=0,
$$
whose solutions are as indicated.
\end{proof}

As an illustration, Proposition \ref{col} explicitly gives the full commutator $[g_i,g_j]$ in $G_{100}(R)$, where 
$2\leq j<i\leq 100$, provided $i+4j-4\geq 101$. This occurs for all $j\geq 21$, all $[g_i,g_{20}]$ with $i\geq 25$,
all $[g_i,g_{19}]$ with $i\geq 29$, and so on. Moreover, in the excepted cases, it gives many
of the beginning terms of the commutator. In the case that concerns us, namely $G_{15}(R)$, 
Proposition \ref{col} helps us obtain the following result.

\section{Commutators}

\begin{prop}\label{formcom} The following equations hold in $G_{15}(R)$:
\begin{eqnarray*}
&&\left[g_3,g_2\right]=x+x^4+x^5-5x^6-x^7+8x^8+16x^9-63x^{10}+59x^{11}+38x^{12}+131x^{13}-1290x^{14}+2556x^{15},\\ 
&&\left[g_4,g_2\right]=x+2x^5+2x^6-8x^8-12x^9+24x^{11}+36x^{12}-64x^{13}-72x^{14}-16x^{15},\\
&&\left[g_5,g_2\right] = x+3x^6+4x^7+2x^8+x^9-17x^{10}-22x^{11}-23x^{12}-7x^{13}+59x^{14}+54x^{15},\\
&&\left[g_6,g_2\right]=x+4x^7+7x^8+6x^9+3x^{10}-24x^{12}-50x^{13}-50x^{14}-58x^{15},\\
&&\left[g_7,g_2\right]=x+5x^8+11x^{9}+13x^{10}+9x^{11}+3x^{12}+x^{13}-37x^{14}-83x^{15},\\
&&\left[g_8,g_2\right]=x+6x^9+16x^{10}+24x^{11}+22x^{12}+12x^{13}+4x^{14},\\
&&\left[g_9,g_2\right]=x+7x^{10}+22x^{11}+40x^{12}+46x^{13}+34x^{14}+16x^{15},\\
&&\left[g_{10},g_2\right]=x+8x^{11}+29x^{12}+62x^{13}+86x^{14}+80x^{15},\\
&&\left[g_{11},g_2\right]=x+9x^{12}+37x^{13}+91x^{14}+148x^{15},\\
&&\left[g_{12},g_2\right]=x+10x^{13}+46x^{14}+128x^{15},\\
&&\left[g_{13},g_2\right]=x+11x^{14}+56x^{15},\\
&&\left[g_{14},g_2\right]=x+12x^{15},\\
&&\left[g_4,g_3\right]=x+x^6+3x^8-7x^9-5x^{10}-9x^{11}+43x^{12}-10x^{13}-70x^{15},\\
&&\left[g_5,g_3\right]=x+2x^7+4x^9-15x^{11}-12x^{13}+15x^{15},\\
&&\left[g_6,g_3\right]=x+3x^8+6x^{10}+2x^{12}-21x^{13}+9x^{14}-45x^{15},\\
&&\left[g_7,g_3\right]=x+4x^9+9x^{11}+8x^{13}-20x^{15},\\
&&\left[g_8,g_3\right]=x+5x^{10}+13x^{12}+17x^{14},\\
&&\left[g_9,g_3\right]=x+6x^{11}+18x^{13}+30x^{15},\\
&&\left[g_{10},g_3\right]=x+7x^{12}+24x^{14},\\
&&\left[g_{11},g_3\right]=x+8x^{13}+31x^{15},\\
&&\left[g_{12},g_3\right]=x+9x^{14},\\
&&\left[g_{13},g_3\right]=x+10x^{15},\\
&&\left[g_{5},g_4\right]=x+x^{8}+6x^{11}-11 x^{12}-14x^{14}-18 x^{15},\\
&&\left[g_{6},g_4\right]=x+2x^{9}+7x^{12}-18 x^{14}-8 x^{15},\\
&&\left[g_{7},g_4\right]=x+3x^{10}+9x^{13},\\
&&\left[g_{8},g_4\right]=x+4x^{11}+12x^{14},\\
&&\left[g_{9},g_4\right]=x+5x^{12}+16x^{15},\\
&&\left[g_{10},g_4\right]=x+6x^{13},\\
&&\left[g_{11},g_4\right]=x+7x^{14},\\
&&\left[g_{12},g_4\right]=x+8x^{15},\\
&&\left[g_{6},g_5\right]=x+x^{10}+10 x^{14}-16x^{15},\\
&&\left[g_{7},g_5\right]=x+2x^{11}+11x^{15},
\end{eqnarray*}
as well as
\begin{eqnarray*}
&&\left[g_{8},g_5\right]=x+3x^{12},\\
&&\left[g_{9},g_5\right]=x+4x^{13},\\
&&\left[g_{10},g_5\right]=x+5x^{14},\\
&&\left[g_{11},g_5\right]=x+6x^{15},\\
&&\left[g_{7},g_6\right]=x+x^{12},\\
&&\left[g_{8},g_6\right]=x+2x^{13}\\
&&\left[g_{9},g_6\right]=x+3x^{14},\\
&&\left[g_{10},g_6\right]=x+4x^{15},\\
&&\left[g_{8},g_7\right]=x+x^{14},\\
&&\left[g_{9},g_7\right]=x+2x^{15},\\
&&\left[g_i,g_j\right]=x\text{ if }i+j-1\geq 16.
\end{eqnarray*}
\end{prop}

\begin{proof} Let $2\leq j<i\leq 15$. Proposition \ref{col} yields 
the case $i+j-1\geq 16$ as well as the case when $i+j-1<16$ and $i+4j-4\geq 16$,
which holds for all $j\geq 4$, as well as for all $[g_i,g_3]$ with $8\leq i\leq 15$ and 
all $[g_i,g_2]$ with $12\leq i\leq 15$. 
It remains to determine the missing terms of the commutators $[g_i,g_3]$, $4\leq i\leq 7$, and
$[g_i,g_2]$, $3\leq i\leq 11$. In these cases, we appeal to (\ref{pup}) and (\ref{pup2}) and solve
for the unknown coefficients of $[g_i,g_j]=x+b_{i+j-1}x^{i+j-1}+b_{i+j}x^{i+j}+\cdots$ in 
\begin{equation}\label{pup3}
g_j*g_i*[g_i,g_j]=g_i*g_j,
\end{equation}
where $b_{i+j-1}=i-j$ by Proposition \ref{col}. %We divide the determination of the remaining coefficients into 2 cases.

\medskip

$\bullet$ $[g_i,g_3]$, $4\leq i\leq 7$. From Proposition \ref{col}, we have
$$
[g_4,g_3]\equiv x+x^6+3x^8-7x^9-5x^{10}-9x^{11}\mod x^{12},\; [g_5,g_3]\equiv x+2x^7+4x^9-15x^{11}\mod x^{13},
$$
$$
[g_6,g_3]\equiv x+3x^8+6x^{10}+2x^{12}-21x^{13}\mod x^{14},\; [g_7,g_3]\equiv x+4x^9+9x^{11}+8x^{13}\mod x^{15},
$$
and we proceed to determine the remaining coefficients. 
Appealing to (\ref{pup}) and (\ref{pup2}) with $j=3$, we see that~(\ref{pup3}) becomes
\begin{eqnarray}
\label{pup4}
(x+x^i+x^3+3x^{i+2}+3x^{2i+1}+x^{3i})*(x+b_{i+2}x^{i+2}+b_{i+3}x^{i+3}+b_{i+4}x^{i+4}+\cdots)=
\nonumber\\x+x^3+x^i+i x^{i+2}+\binom{i}{2}x^{i+4}+
\binom{i}{3}x^{i+6}+\binom{i}{4}x^{i+8}+\cdots
\end{eqnarray}
We first use (\ref{pup4}) to find $b_{i+8}$ when $4\leq i\leq 7$. We see that
\begin{eqnarray*}
b_{i+8}+ib_{i+2}+3b_{i+6}+3&=&\binom{i}{4}\text{ if }i=7,\\
b_{i+8}+3b_{i+6}&=&\binom{i}{4}\text{ if }i=6,\\
b_{i+8}+ib_{i+4}+3b_{i+6}+3(i+2)b_{i+2}&=&\binom{i}{4}\text{ if }i=5,\\
b_{i+8}+ib_{i+5}+3b_{i+6}+1&=&\binom{i}{4}\text{ if }i=4.
\end{eqnarray*}
Thus
$$
b_{15}=-20\text{ if }i=7,\; b_{14}=9\text{ if }i=6,\; b_{13}=-12\text{ if }i=5,\; b_{12}=43\text{ if }i=4.
$$
We next utilize (\ref{pup4}) to determine $b_{i+9}$ for $4\leq i\leq 6$, obtaining
\begin{eqnarray*}
b_{i+9}+ib_{i+4}+3b_{i+7}+3(i+2)b_{i+2}&=&0\text{ if }i=6,\\
b_{i+9}+ib_{i+5}+3b_{i+7}&=&0\text{ if }i=5,\\
b_{i+9}+ib_{i+6}+3b_{i+7}+3 b_{(i+8)/2}^2+3(i+2)b_{i+4}&=&0\text{ if }i=4.\\
\end{eqnarray*}
Therefore
$$
b_{15}=-45\text{ if }i=6,\; b_{14}=0\text{ if }i=5,\; b_{13}=-10\text{ if }i=4.
$$
We next appeal to (\ref{pup4}) to find $b_{i+10}$ if $4\leq i\leq 5$, deriving
\begin{eqnarray*}
b_{i+10}+ib_{i+6}+3b_{i+8}+3 b_{(i+9)/2}^2+3(i+2)b_{i+4}+1&=&\binom{i}{5}\text{ if }i=5,\\
b_{i+10}+ib_{i+7}+6b_{i+2}^2+3b_{i+8}+3(i+2)b_{i+5}+3(2i+1)b_{i+2}&=&0\text{ if }i=4.\\
\end{eqnarray*}
Hence
$$
b_{15}=15\text{ if }i=5,\; b_{14}=0\text{ if }i=4.
$$
We finally use (\ref{pup4}) to determine $b_{15}$ if $i=4$. We find that
$$
b_{15}+4b_{12}+3b_{13}+12 b_6b_7+3 b_{7}^2+6b_6b_8+18b_{10}+27b_7=0,
$$
so $b_{15}=-70$.

\medskip

$\bullet$ $[g_i,g_2]$, $3\leq i\leq 11$. From Proposition \ref{col}, we have
$$
[g_3,g_2]\equiv x+x^4+x^5-5x^6\mod x^{7},\; [g_4,g_2]\equiv x+2x^5+2x^6\mod x^{8},
$$
$$
[g_5,g_2]\equiv x+3x^6+4x^{7}+2x^8\mod x^{9},\; [g_6,g_2]\equiv x+4x^7+7x^{8}+6x^9\mod x^{10},
$$
$$
[g_7,g_2]\equiv x+5x^8+11x^{9}+13x^{10}\mod x^{11},\; [g_8,g_2]\equiv x+6x^9+16x^{10}+24x^{11}\mod x^{12},
$$
$$
[g_9,g_2]\equiv x+7x^{10}+22x^{11}+40x^{12}\mod x^{13},\;
[g_{10},g_2]\equiv x+8x^{11}+29x^{12}+62x^{13}\mod x^{14},
$$
$$
[g_{11},g_2]\equiv x+9x^{12}+37x^{13}+91x^{14}\mod x^{15},
$$
and we proceed to determine the remaining coefficients. Making use of (\ref{pup}) and (\ref{pup2}) we find that~(\ref{pup3}) translates as follows:
\begin{eqnarray}
\label{pup5}
(x+x^i+x^2+2x^{i+1}+x^{2i})*(x+b_{i+1}x^{i+1}+b_{i+2}x^{i+2}+b_{i+3}x^{i+3}+\cdots)=
\nonumber\\x+x^2+x^i+i x^{i+1}+\binom{i}{2}x^{i+2}+
\binom{i}{3}x^{i+3}+\binom{i}{4}x^{i+4}+\cdots
\end{eqnarray}
We first use (\ref{pup5}) to find $b_{i+4}$ when $3\leq i\leq 11$. We see that
\begin{eqnarray*}
b_{i+4}+2b_{i+3}+ib_{i+2}+2(i+1)b_{i+1}&=&0\text{ if }i=3,\\
b_{i+4}+2b_{i+3}+ib_{i+1}+1&=&\binom{i}{4}\text{ if }i=4,\\
b_{i+4}+2b_{i+3}&=&\binom{i}{4}\text{ if }i>4.
\end{eqnarray*}
Thus
$$
b_{7}=-1\text{ if }i=3,\; b_{8}=-8\text{ if }i=4,\; b_{9}=1\text{ if }i=5,\; b_{10}=3\text{ if }i=6,\;
b_{11}=9\text{ if }i=7,
$$
$$
b_{12}=22\text{ if }i=8,\; b_{13}=46\text{ if }i=9,\; b_{14}=86\text{ if }i=10,\; b_{15}=148\text{ if }i=11.\;
$$
We next utilize (\ref{pup5}) to determine $b_{i+5}$ for $3\leq i\leq 10$, obtaining
\begin{eqnarray*}
b_{i+5}+2b_{i+4}+ib_{i+3}+2(i+1)b_{i+2}+b_{i+1}^2&=&0\text{ if }i=3,\\
b_{i+5}+2b_{i+4}+ib_{i+2}+2(i+1)b_{i+1}&=&0\text{ if }i=4,\\
b_{i+5}+2b_{i+4}+ib_{i+1}+1&=&\binom{i}{5}\text{ if }i=5,\\
b_{i+5}+2b_{i+4}&=&\binom{i}{5}\text{ if }i>5.
\end{eqnarray*}
Therefore
$$
b_{8}=8\text{ if }i=3,\; b_{9}=-12\text{ if }i=4,\; b_{10}=-17\text{ if }i=5,
$$
$$
b_{11}=0\text{ if }i=6,\; b_{12}=3\text{ if }i=7,\; b_{13}=12\text{ if }i=8,
$$
$$
b_{14}=34\text{ if }i=9,\; b_{15}=80\text{ if }i=10.
$$
We next appeal to (\ref{pup5}) to find $b_{i+6}$ if $3\leq i\leq 9$, deriving
\begin{eqnarray*}
b_{i+6}+2b_{i+5}+ib_{i+4}+3b_{i+1}^2+2b_{i+1}b_{i+2}+2(i+1)b_{i+3}+2i b_{i+1}&=&0\text{ if }i=3,\\
b_{i+6}+2b_{i+5}+b_{i+1}^2+ib_{i+3}+2(i+1)b_{i+2}&=&0\text{ if }i=4,\\
b_{i+6}+2b_{i+5}+ib_{i+2}+2(i+1)b_{i+1}&=&0\text{ if }i=5,\\
b_{i+6}+2b_{i+5}+ib_{i+1}+1&=&\binom{i}{6}\text{ if }i=6,\\
b_{i+6}+2b_{i+5}&=&\binom{i}{6}\text{ if }i>6.
\end{eqnarray*}
Hence
$$
b_{9}=16\text{ if }i=3,\; b_{10}=0\text{ if }i=4,\; b_{11}=-22\text{ if }i=5,\; b_{12}=-24\text{ if }i=6,
$$
$$
b_{13}=1\text{ if }i=7,\; b_{14}=4\text{ if }i=8,\; b_{15}=16\text{ if }i=9.
$$
We next resort to (\ref{pup5}) to find $b_{i+7}$ if $3\leq i\leq 8$, obtaining
\begin{eqnarray*}
b_{i+7}+2b_{i+6}+ib_{i+5}+6b_{i+1}b_{i+2}+2b_{i+1}b_{i+3}+b_{i+2}^2+2(i+1)b_{i+4}+
12b_{i+1}^2+2i b_{i+2}&=&0\text{ if }i=3,\\
b_{i+7}+2b_{i+6}+ib_{i+4}+2b_{i+1}b_{i+2}+2(i+1)b_{i+3}&=&0\text{ if }i=4,\\
b_{i+7}+2b_{i+6}+ib_{i+3}+b_{i+1}^2+2(i+1)b_{i+2}&=&0\text{ if }i=5,\\
b_{i+7}+2b_{i+6}+ib_{i+2}+2(i+1)b_{i+1}&=&0\text{ if }i=6,\\
b_{i+7}+2b_{i+6}+ib_{i+1}+1&=&\binom{i}{7}\text{ if }i=7,\\
b_{i+7}+2b_{i+6}&=&\binom{i}{7}\text{ if }i>7.
\end{eqnarray*}
Therefore,
$$
b_{10}=-63\text{ if }i=3,\; b_{11}=24\text{ if }i=4,\; b_{12}=-23\text{ if }i=5,
$$
$$
b_{13}=-50\text{ if }i=6,\; b_{14}=-37\text{ if }i=7,\; b_{15}=0\text{ if }i=8.
$$
We next invoke (\ref{pup5}) to find $b_{i+8}$ if $3\leq i\leq 7$, which results in
\begin{eqnarray*}
b_{i+8}+2(b_{i+7}+b_{i+1}b_{i+4}+b_{i+2}b_{i+3})+ib_{i+6}+6b_{i+1}b_{i+3}+3b_{i+2}^2+2(i+1)b_{i+5}+
24b_{i+1}b_{i+2}+2i b_{i+3}&=&0\text{ if }i=3,\\
b_{i+8}+2b_{i+7}+ib_{i+5}+6b_{i+1}^2+b_{i+2}^2+2b_{i+1}b_{i+3}+2(i+1)b_{i+4}+2i b_{i+1}&=&0\text{ if }i=4,\\
b_{i+8}+2b_{i+7}+ib_{i+4}+2(i+1)b_{i+3}+2b_{i+1}b_{i+2}&=&0\text{ if }i=5,\\
b_{i+8}+2b_{i+7}+ib_{i+3}+2(i+1)b_{i+2}+b_{i+1}^2&=&0\text{ if }i=6,\\
b_{i+8}+2b_{i+7}+ib_{i+2}+2(i+1)b_{i+1}&=&0\text{ if }i=7.
\end{eqnarray*}
Thus,
$$
b_{11}=59\text{ if }i=3,\; b_{12}=36\text{ if }i=4,\; b_{13}=-7\text{ if }i=5,\;
b_{14}=-50\text{ if }i=6,\; b_{15}=-83\text{ if }i=7.
$$
We next make use (\ref{pup5}) to determine $b_{i+9}$ if $3\leq i\leq 6$, obtaining
\begin{eqnarray*}
b_{i+9}+ib_{i+7}+6b_{i+2}b_{i+3}+6b_{i+1}b_{i+4}+b_{i+1}^3+2b_{i+2}b_{i+4}+\\
2b_{i+1}b_{i+5}+b_{i+3}^2+2b_{i+8}+2(i+1)b_{i+6}
+24b_{i+1}b_{i+3}+12b_{i+2}^2+2i b_{i+4}+15b_{i+1}^2
&=&0\text{ if }i=3,\\
b_{i+9}+ib_{i+6}+12b_{i+1}b_{i+2}+2(b_{i+1}b_{i+4}+b_{i+2}b_{i+3}+b_{i+8})+2(i+1) b_{i+5}+
20b_{i+1}^2+2ib_{i+2}&=&0\text{ if }i=4,\\
b_{i+9}+ib_{i+5}+2b_{i+1}b_{i+3}+b_{i+2}^2+2b_{i+8}+2(i+1)b_{i+4}&=&0\text{ if }i=5,\\
b_{i+9}+ib_{i+4}+2b_{i+1}b_{i+2}+2b_{i+8}+2(i+1)b_{i+3}&=&0\text{ if }i=6.
\end{eqnarray*}
Hence,
$$
b_{12}=38\text{ if }i=3,\; b_{13}=-64\text{ if }i=4,\; b_{14}=59\text{ if }i=5,\;
b_{15}=-58\text{ if }i=6.
$$
We next refer to (\ref{pup5}) to find $b_{i+10}$ if $3\leq i\leq 5$, deriving
\begin{eqnarray*}
b_{i+10}+ib_{i+8}+3b_{i+3}^2+ 6(b_{i+2}b_{i+4}+b_{i+1}b_{i+5})+3b_{i+1}^2b_{i+2}+2b_{i+9}+\\
2(b_{i+3}b_{i+4}+b_{i+2}b_{i+5}+b_{i+1}b_{i+6})+2i b_{i+5}+30b_{i+1}b_{i+2}+2(i+1) b_{i+7}+\\
24(b_{i+2}b_{i+3}+b_{i+1}b_{i+4})+8b_{i+1}^3&=&0\text{ if }i=3,\\
b_{i+10}+ib_{i+7}+6b_{i+2}^2+12b_{i+1}b_{i+3}+b_{i+3}^2+2(b_{i+2}b_{i+4}+b_{i+1}b_{i+5})+\\
2b_{i+9}+2i b_{i+3}+2(i+1) b_{i+6}+40b_{i+1}b_{i+2}&=&0\text{ if }i=4,\\
b_{i+10}+ib_{i+6}+2b_{i+9}+10b_{i+1}^2+2(b_{i+2}b_{i+3}+b_{i+1}b_{i+4})+2ib_{i+1}+2(i+1)b_{i+5}&=&0\text{ if }i=5.
\end{eqnarray*}
Therefore,
$$
b_{13}=131\text{ if }i=3,\; b_{14}=-72\text{ if }i=4,\; b_{15}=54\text{ if }i=5.
$$
We next go back to (\ref{pup5}) to determine $b_{i+11}$ if $3\leq i\leq 4$, which results in
\begin{eqnarray*}
b_{i+11}+ib_{i+9}+ 6(b_{i+3}b_{i+4}+b_{i+2}b_{i+5}+b_{i+1}b_{i+6})+3(b_{i+1}b_{i+2}^2+b_{i+1}^2b_{i+3})+\\
b_{i+4}^2+2(b_{i+3}b_{i+5}+b_{i+2}b_{i+6}+b_{i+1}b_{i+7})+2b_{i+10}+8b_{i+8}+24(b_{i+1}b_{i+5}+b_{i+2}b_{i+4})+\\
12b_{i+3}^2+24b_{i+1}^2b_{i+2}+6b_{i+6}+30b_{i+1}b_{i+3}+15b_{i+2}^2&=&0\text{ if }i=3,\\
b_{i+11}+ib_{i+8}+12(b_{i+2}b_{i+3}+b_{i+1}b_{i+4})+2(b_{i+3}b_{i+4}+b_{i+2}b_{i+5}+b_{i+1}b_{i+6})+\\
2b_{i+10}+2(i+1)b_{i+7}+40b_{i+1}b_{i+3}+20b_{i+2}^2+2i b_{i+4}&=&0\text{ if }i=4.
\end{eqnarray*}
Hence,
$$
b_{14}=-1290\text{ if }i=3,\; b_{15}=-16\text{ if }i=4.
$$
We apply (\ref{pup5}) one last time to find $b_{15}$ if $i=3$, deriving
\begin{eqnarray*}
b_{15}+3b_{13}+ 6(b_{4}b_{10}+b_{5}b_{9}+b_{6}b_{8}+b_{4}b_{5}b_6)+3(b_7^2+b_4^2b_7)+b_5^3+\\
2(b_{14}+b_{4}b_{11}+b_{5}b_{10}+b_{6}b_{9}+b_{7}b_{8})+8b_{12}+\\
24(b_{4}b_{9}+b_{5}b_{8}+b_{6}b_{7}+b_{4}^2 b_{6}+b_{4} b_{5}^2)+6 b_{10}+30(b_{4}b_{7}+b_{5}b_{6})+20b_4^3&=&0\text{ if }i=3.
\end{eqnarray*}
Thus,
$$
b_{15}=2556\text{ if }i=3.
$$
\end{proof}

\section{Commutators in terms of generators}\label{nt}

We fix $n\geq 2$ throughout this section. Proposition \ref{com1} ensures that when $2\leq u\leq v\leq n$ and $u+v-1\geq n+1$, the subgroups $H_u$ and~$H_v$ of $G_n(R)$ satisfy $[H_u,H_v]=1$. For our purposes, we need to the following more detailed result
below. 

\begin{lemma}\label{rad} Suppose that $2\leq u\leq v\leq n$ and $u+v-1\geq n+1$. Then the following identity holds in $G_n(R)$:
$$(x+a_u x^{u}+\cdots+a_n x^{n})*(x+b_v x^{v}+\cdots+b_n x^{n})=x+a_ux^u+\cdots+a_{v-1}x^{v-1}+(a_v+b_v) x^{v}+\cdots+
(a_n+b_n)  x^{n},$$
where $a_ux^u+\cdots+a_{v-1}x^{v-1}$ is to be erased if $u=v$.
\end{lemma}

\begin{proof} The group operation in $G_n(R)$ together with $u+v-1\geq n+1$ yield the desired result.
\end{proof}

We assume until further notice that $R=\Z$. All calculations in this section are performed in $G(\Z)$,
but the results remain valid when $R$
is an arbitrary ring with $1\neq 0$, via the natural ring homomorphism $\Z\to R$, which yields a natural group homomorphism
$G(\Z)\to G(R)$, except that when $R$ does not have characteristic 0, the unique solutions found below in the case $R=\Z$
will not be unique when $R$ has prime characteristic. We will have more to say about this lack of uniqueness in the next section.

For any $m\in\Z$, we define $\mu(m)=\binom{m}{2}$ if $m\geq 0$, $\mu(m)=\binom{-m+1}{2}$ if $m<0$.

\begin{lemma}\label{rad5} Suppose that $k\geq 2$ and $3k-2\geq n+1$. Then for any $m\in\Z$ and $a\in R$, we have
\begin{equation}\label{maxmart}
(x+ax^{k})^{(m)}=x+ma x^k+\mu(m)a^2 k x^{2k-1}.
\end{equation}
\end{lemma}

\begin{proof} Clearly (\ref{maxmart}) holds when $m=0$. Assume that
(\ref{maxmart}) is true for some $m\geq 0$. Then
$$
\begin{aligned}
(x+ax^{k})^{(m+1)}&=(x+a x^k)*(x+ma x^k+\mu(m)a^2 k x^{2k-1})\\
&=x+ma x^k+\mu(m)a^2 k x^{2k-1}+a(x+ma x^k+\mu(m)a^2 k x^{2k-1})^k\\
&=x+ma x^k+\mu(m)a^2 k x^{2k-1}+a x^k+ ma^2k x^{2k-1}\\
&=x+(m+1)a+\mu(m+1)a^2k x^{2k-1}.
\end{aligned}
$$
It follows by induction that (\ref{maxmart}) is valid for all $m\geq 0$. Moreover, from
$$
(x+a x^k)*(x-a x^k+a^2 k x^{2k-1})=x,
$$
we infer that $(x+a x^k)^{(-1)}=x-a x^k+a^2 k x^{2k-1}$. Thus (\ref{maxmart}) holds when $m=-1$ and we also have
$(x+a x^k)^{(-1)}=(x-ax^k)*(x+a^2 k x^{2k-1})$ by Lemma \ref{rad}. Given any $m>0$, the previous case 
and Lemma \ref{rad} now yield
$$
\begin{aligned}
(x+ax^{k})^{(-m)}&=(x-ax^k)^{(m)}*(x+a^2 k x^{2k-1})^{(m)}\\
&=(x-ma x^k+\mu(m)a^2 k x^{2k-1})*(x+m a^2 k x^{2k-1})\\
&=x-ma x^k+\mu(m+1)a^2 k x^{2k-1},
\end{aligned}
$$
which confirms (\ref{maxmart}) for all negative integers.
\end{proof}

\begin{prop}\label{rad6} Suppose that $2\leq k\leq n$ and $2k>n+1$. Then given any $b_k,\dots,b_n\in\Z$,
the unique solution to
\begin{equation}
\label{real}
g_k^{(a_k)}*\cdots*g_n^{(a_n)}=x+b_k x^{k}+\cdots+b_n x^{n}
\end{equation}
for integers $a_k,\dots,a_n\in\Z$ is given by $a_k=b_k,\dots,a_n=b_n$.
\end{prop}

\begin{proof} By Lemma \ref{rad}, we have $(x+a_k x^{k})*\cdots*(x+a_n x^{n})=x+a_k x^{k}+\cdots+a_n x^{n}$,
while Lemma~\ref{rad5} ensures that $g_i^{(a_i)}=x+a_i x^{i}$ for each $k\leq i\leq n$.
\end{proof}

\begin{prop}\label{rad7} Suppose that $2\leq k$ and $2k=n+1$. Then given any $b_k,\dots,b_n\in\Z$,
the unique solution to (\ref{real}) for integers $a_k,\dots,a_n\in\Z$ is given by $a_i=b_i$ for $k\leq i<n$ and
$a_n=b_n-\mu(a_k)k$.
\end{prop}

\begin{proof} Lemmas \ref{rad} and \ref{rad5} yield 
$g_{k+1}^{(a_{k+1})}*\cdots*g_{n}^{(a_{n})}=x+a_{k+1} x^{k+1}+\cdots+a_n x^{n}$,
while Lemma \ref{rad5} guarantees that $g_k^{(a_k)}=x+a_k x^k+\mu(a_k)k x^{2k-1}$. A second application of Lemma \ref{rad} now yields
the desired result.
\end{proof}

\begin{prop}\label{rad8} Suppose that $8\leq 2k<n+1$ and $2(k+1)\geq n+1$. Then 
given any $b_k,\dots,b_n\in\Z$,
the unique solution to (\ref{real}) for integers $a_k,\dots,a_n\in\Z$ is given by $a_i=b_i$ for $k\leq i<2k-1$, as well as
$$
a_{2k-1}=b_{2k-1}-\mu(a_k)k,\; a_{2k}=b_{2k}-ka_k a_{k+1},\; a_{2k+1}=b_{2k+1}-ka_k a_{k+2}-\mu(a_{k+1})(k+1).
$$
\end{prop}

\begin{proof} By Lemmas \ref{rad} and \ref{rad5}, we have 
$g_{k+2}^{(a_{k+2})}*\cdots*g_{n}^{(a_{n})}=x+a_{k+2} x^{k+2}+\cdots+a_n x^{n}$,
while Lemma~\ref{rad5} ensures that $$g_{k+1}^{(a_{k+1})}=x+a_{k+1} x^{k+1}+\mu(a_{k+1})(k+1) x^{2k+1},\;
g_k^{(a_k)}=x+a_k x^k+\mu(a_k)k x^{2k-1}.$$ 
Using these identities, an elementary calculation yields
$$
\begin{aligned}
g_k^{(a_k)}*g_{k+1}^{(a_{k+1})}*g_{k+2}^{(a_{k+2})}*\cdots*g_{n}^{(a_{n})}&=x+a_k+a_{k+1} x^{k+1}+\cdots+a_n x^{n}+\\
\mu(a_k)k x^{2k-1}+k a_ka_{k+1} x^{2k}+(\mu(a_{k+1})(k+1)+ k a_k a_{k+2})x^{2k+1}.
\end{aligned}
$$
\end{proof}

\begin{prop}\label{rad9} Suppose that $14\leq 2(k+1)<n+1$ and $2(k+2)\geq n+1$. Then 
given any $b_k,\dots,b_n\in\Z$,
the unique solution to (\ref{real}) for integers $a_k,\dots,a_n\in\Z$ is given by $a_i=b_i$ for $k\leq i<2k-1$, as well as
$$
a_{2k-1}=b_{2k-1}-\mu(a_k)k,\; a_{2k}=b_{2k}-ka_k a_{k+1},\; a_{2k+1}=b_{2k+1}-ka_k a_{k+2}-\mu(a_{k+1})(k+1),
$$
$$
a_{2k+2}=b_{2k+2}-ka_ka_{k+3}-(k+1)a_{k+1}a_{k+2},\; a_{2k+3}=b_{2k+3}-ka_k a_{k+4}-(k+1)a_{k+1}a_{k+3}-\mu(a_{k+2})(k+2).
$$
\end{prop}

\begin{proof} By Lemmas \ref{rad} and \ref{rad5}, we have 
$g_{k+3}^{(a_{k+3})}*\cdots*g_{n}^{(a_{n})}=x+a_{k+3} x^{k+3}+\cdots+a_n x^{n}$,
while Lemma~\ref{rad5} guarantees that $g_{k+i}^{(a_{k+i})}=x+a_{k+i} x^{k+i}+\mu(a_{k+i})(k+i) x^{2(k+i)-1}$ when $0\leq i\leq 2$.
Using these identities, an elementary calculation yields
$$
\begin{aligned}
g_k^{(a_k)}*\cdots*g_{n}^{(a_{n})}&=x+a_kx^k+a_{k+1} x^{k+1}+\cdots+a_n x^{n}\\
&+\mu(a_k)k x^{2k-1}+k a_ka_{k+1} x^{2k}+(\mu(a_{k+1})(k+1)+ k a_k a_{k+2})x^{2k+1}\\
&+(k a_k a_{k+3}+(k+1)a_{k+1}a_{k+2})x^{2k+2}\\
&+(k a_k a_{k+4}+(k+1)a_{k+1}a_{k+3}+\mu(a_{k+2})(k+2))x^{2k+3}.
\end{aligned}
$$
\end{proof}

\section{Commutators in terms of generators}\label{nt2}

\begin{prop}\label{formcom2} The following equations hold in $G_{15}(R)$:
\begin{eqnarray*}
&&\left[g_3,g_2\right]=g_4*g_5*g_6^{(-5)}*g_7^{(-1)}*g_8^{(4)}*g_9^{(36)}*g_{10}^{(-34)}*
g_{11}^{(-42)}*g_{12}^{(-162)}*g_{13}^{(360)}*g_{14}^{(-322)}*g_{15}^{(1590)},\\
&&\left[g_4,g_2\right]=g_5^{(2)}*g_6^{(2)}*g_8^{(-8)}*g_9^{(-17)}*
g_{10}^{(-20)}*g_{11}^{(18)}*g_{12}^{(116)}*g_{13}^{(192)}*g_{14}^{(242)}*g_{15}^{(-384)},\\
&&\left[g_5,g_2\right]=
g_6^{(3)}*g_7^{(4)}*g_8^{(2)}*g_9*g_{10}^{(-17)}*g_{11}^{(-40)}*g_{12}^{(-95)}*g_{13}^{(-85)}*g_{14}^{(-15)}*g_{15}^{(324)},\\
&&\left[g_6,g_2\right]=g_7^{(4)}*g_8^{(7)}*g_9^{(6)}*g_{10}^{(3)}*g_{12}^{(-24)}*g_{13}^{(-92)}*
g_{14}^{(-246)}*g_{15}^{(-394)},\\
&&\left[g_7,g_2\right]=g_8^{(5)}*g_9^{(11)}*g_{10}^{(13)}*g_{11}^{(9)}*g_{12}^{(3)}*g_{13}
*g_{14}^{-37}*g_{15}^{-163},\\ 
&&\left[g_8,g_2\right]=g_9^{(6)}*g_{10}^{(16)}*g_{11}^{(24)}*g_{12}^{(22)}*g_{13}^{(12)}*
g_{14}^{(4)},\\ 
&&\left[g_9,g_2\right]=g_{10}^{(7)}*g_{11}^{(22)}*g_{12}^{(40)}*g_{13}^{(46)}*g_{14}^{(34)}*g_{15}^{(16)},\\ 
&&\left[g_{10},g_2\right]=g_{11}^{(8)}*g_{12}^{(29)}*g_{13}^{(62)}*g_{14}^{(86)}*g_{15}^{(80)},\\
&&\left[g_{11},g_2\right]=g_{12}^{(9)}*g_{13}^{(37)}*g_{14}^{(91)}*g_{15}^{(148)},\\
&&\left[g_{12},g_2\right]=g_{13}^{(10)}*g_{14}^{(46)}*g_{15}^{(128)},\\
&&\left[g_{13},g_2\right]=g_{14}^{(11)}*g_{15}^{(56)},\\
&&\left[g_{14},g_2\right]=g_{15}^{(12)},\\
&&\left[g_4,g_3\right]=g_6*g_8^{(3)}*g_9^{(-7)}*g_{10}^{(-5)}*g_{11}^{(-9)}*g_{12}^{(43)}*
g_{13}^{(-28)}*g_{14}^{(42)}*g_{15}^{(-64)},\\ 
&&\left[g_5,g_3\right]=g_7^{(2)}*g_9^{(4)}*g_{11}^{(-15)}*g_{13}^{(-19)}*g_{15}^{(-41)},\\ 
&&\left[g_6,g_3\right]=g_8^{(3)}*g_{10}^{(6)}*g_{12}^{(2)}*g_{13}^{(-21)}*g_{14}^{(9)}*g_{15}^{(-69)},\\ 
&&\left[g_7,g_3\right]=g_9^{(4)}*g_{11}^{(9)}*g_{13}^{(8)}*g_{15}^{(-20)},\\
&&\left[g_8,g_3\right]=g_{10}^{(5)}*g_{12}^{(13)}*g_{14}^{(17)},\\
&&\left[g_9,g_3\right]=g_{11}^{(6)}*g_{13}^{(18)}*g_{15}^{(30)},\\
&&\left[g_{10},g_3\right]=g_{12}^{(7)}*g_{14}^{(24)},\\
&&\left[g_{11},g_3\right]=g_{13}^{(8)}*g_{15}^{(31)},\\
&&\left[g_{12},g_3\right]=g_{14}^{(9)},\\
&&\left[g_{13},g_3\right]=g_{15}^{(10)},\\
&&\left[g_{5},g_4\right]=g_{8}*g_{11}^{(6)}*g_{12}^{(-11)}*g_{14}^{(-14)}*g_{15}^{(-18)},\\ 
&&\left[g_{6},g_4\right]=g_{9}^{(2)}*g_{12}^{(7)}*g_{14}^{(-18)}*g_{15}^{(-8)},\\  
&&\left[g_{7},g_4\right]=g_{10}^{(3)}*g_{13}^{(9)},\\
&&\left[g_{8},g_4\right]=g_{11}^{(4)}*g_{14}^{(12)},\\
&&\left[g_{9},g_4\right]=g_{12}^{(5)}*g_{15}^{(16)},\\
&&\left[g_{10},g_4\right]=g_{13}^{(6)},\\
&&\left[g_{11},g_4\right]=g_{14}^{(7)},\\
&&\left[g_{12},g_4\right]=g_{15}^{(8)},\\
&&\left[g_{6},g_5\right]=g_{10}*g_{14}^{(10)}*g_{15}^{(-16)},\\
&&\left[g_{7},g_5\right]=g_{11}^{(2)}*g_{15}^{(11)},\\
&&\left[g_{8},g_5\right]=g_{12}^{(3)},\\ 
&&\left[g_{9},g_5\right]=g_{13}^{(4)},\\
&&\left[g_{10},g_5\right]=g_{14}^{(5)},\\
&&\left[g_{11},g_5\right]=g_{15}^{(6)},
\end{eqnarray*}
as well as
\begin{eqnarray*}
&&\left[g_{7},g_6\right]=g_{12},\\
&&\left[g_{8},g_6\right]=g_{13}^{(2)}\\
&&\left[g_{9},g_6\right]=g_{14}^{(3)},\\
&&\left[g_{10},g_6\right]=g_{15}^{(4)},\\
&&\left[g_{8},g_7\right]=g_{14},\\
&&\left[g_{9},g_7\right]=g_{15}^{(2)},\\
&&\left[g_i,g_j\right]=x\text{ if }i+j-1\geq 16.
\end{eqnarray*}
\end{prop}

\begin{proof} Let $2\leq j<i\leq n$ and set $k=i+j-1$ and $n=15$. If $k\geq 16$, we apply Proposition \ref{formcom}.
If $9\leq k\leq 15$, we appeal to Propositions \ref{formcom} and \ref{rad6}. The case $k=8$
follows from Propositions \ref{formcom} and \ref{rad7}. When $k=7$,
we utilize Propositions \ref{formcom} and \ref{rad8}. If $k=6$,
we resort to Propositions \ref{formcom} and \ref{rad9}.

Suppose next that $k=5$, that is, $j=2$ and $i=4$. We have
$
[g_4,g_2]=x+b_5 x^5+\cdots+b_{15}x^{15},
$
where $b_5,\dots,b_{15}$ are as indicated in Proposition \ref{formcom}. In particular,  $b_5=2$,
so by Lemma \ref{nat} there exist integers $d_6,\dots,d_{15}$ such that
$$
[g_4,g_2]=g_5^{(2)}*(x+d_6x^6+\cdots+d_{15}x^{15}).
$$
Now $g_5^{(2)}=x+2x^5+5x^9+10x^{13}$, so
$$
\begin{aligned}
\left [g_4,g_2\right]&=x+d_6x^6+\cdots+d_{15}x^{15}+2(x+d_6x^6+\cdots+d_{15}x^{15})^5\\
&+5(x+d_6x^6+\cdots+d_{15}x^{15})^9+10(x+d_6x^6+\cdots+d_{15}x^{15})^{13}\\
&=x+d_6x^6+\cdots+d_{15}x^{15}+2(x^5+5(d_6x^{10}+d_7x^{11}+\cdots+d_{11}x^{15})\\
&+20d_6^2 x^{15}+5 x^9+45d_6 x^{14}+45d_7 x^{15}+10x^{13},
\end{aligned}
$$
which yields: $d_i=b_i$ for $6\leq i\leq 8$; $d_9+5=b_9$; $d_{i}+10d_{i-4}=b_{i}$ for $10\leq i\leq 12$; and
$$
d_{13}+10d_{9}+10=b_{13},\; d_{14}+10d_{10}+45 d_6=b_{14},\; d_{15}+10d_{11}+45d_7+20 d_6^2=b_{15}.
$$
Solving for $d_6,\dots,d_{15}$ gives
$$
d_6=2,d_7=0, d_8=-8, d_9=-17, d_{10}=-20, d_{11}=24, 
$$
$$
d_{12}=116, d_{13}=96, d_{14}=38,  d_{15}=-336.
$$
We next find integers $a_6,\dots,a_{15}$ such that
$$
g_6^{(a_6)}*\cdots g_{15}^{(a_{15})}=x+d_6x^6+\cdots+d_{15}x^{15}
$$
by means of Proposition \ref{rad9}, which yields the stated result.

Suppose finally that $k=4$, that is, $j=2$ and $i=3$. We have
$
[g_4,g_2]=x+b_5 x^5+\cdots+b_{15}x^{15},
$
where $b_5,\dots,b_{15}$ are as indicated in Proposition \ref{formcom}. In particular, $b_5=b_6=1$,
so by Lemma \ref{nat} there exist integers $d_5,\dots,d_{15}$ such that
$$
[g_3,g_2]=g_4*(x+d_5x^5+d_6x^6+\cdots+d_{15}x^{15}),
$$
that is,
$$
\begin{aligned}
\left [g_3,g_2\right]&=x+d_5x^5+d_6x^6+\cdots+d_{15}x^{15}+(x+d_5x^5+d_6x^6+\cdots+d_{15}x^{15})^4\\
&=x+d_5x^5+d_6x^6+\cdots+d_{15}x^{15}+x^4+4(d_5x^8+d_6x^9+\cdots+d_{12}x^{15}) \\
&+6(d_5^2 x^{12}+d_6^2 x^{14})+12(d_5d_6 x^{13}+d_5d_7 x^{14}+(d_5d_8+d_6d_7) x^{15}),
\end{aligned}
$$
so that $d_i=b_i$ for $5\leq i\leq 7$; $d_{i}+4d_{i-3}=b_{i}$ for $8\leq i\leq 11$; and
$$
d_{12}+4d_9+6d_{5}^2=b_{12},\; d_{13}+4d_{10}+12d_5d_6=b_{13},
$$
$$ 
d_{14}+4d_{11}+6d_6^2+12d_5d_7=b_{14},\; d_{15}+4d_{12}+12(d_5d_8+d_6d_7)=b_{15}.
$$
Solving for $d_5,\dots,d_{15}$ with the aid of Proposition \ref{formcom} yields
$$
d_5=1, d_6=-5,d_7=-1, d_8=4, d_9=36, d_{10}=-59, d_{11}=43, 
$$
$$
d_{12}=-112, d_{13}=427, d_{14}=-1600,  d_{15}=2896.
$$
By Lemma \ref{nat} there exist integers $c_6,\dots,c_{15}$ such that
$$
\begin{aligned}
x+d_5x^5+\cdots+d_{15}x^{15}&=g_5*(x+c_6x^6+\cdots+c_{15}x^{15})\\
&=x+c_6x^6+\cdots+c_{15}x^{15}+(x+c_6x^6+\cdots+c_{15}x^{15})^5\\
&=x+c_6x^6+\cdots+c_{15}x^{15}+x^5+5(c_6 x^{10}+\cdots+c_{11}x^{15})+10 c_6^2 x^{15},
\end{aligned}
$$
so that $c_i=d_i$ for $6\leq i\leq 9$, $c_i+5c_{i-4}=d_i$ for $1\leq i\leq 14$, and $c_{15}+5c_{11}+10c_6^2=d_{15}$.
Solving for $d_6,\dots,c_{15}$,  yields
$$
c_6=-5, c_7=-1, c_8=4, c_9=36, c_{10}=-34, c_{11}=48, 
$$
$$
c_{12}=-132, c_{13}=247, c_{14}=-1430,  c_{15}=2406.
$$
We next find integers $a_6,\dots,a_{15}$ such that
$$
g_6^{(a_6)}*\cdots g_{15}^{(a_{15})}=x+c_6x^6+\cdots+c_{15}x^{15}
$$
by means of Proposition \ref{rad9}, which yields the stated result.
\end{proof}

\section{The second half of the power-commutator presentation}\label{sa}

%We assume in this section that $R=\Z/\Z_p$, but our calculations
%will remain valid when $R$ is an arbitrary ring with $1\neq 0$ of prime characteristic $p$,
%via the natural homomorphism $G(p)\to G(R)$. 

We proceed to use Proposition \ref{formcom2}
obtain the second half of the PCP of $G_{15}(p)$, namely (\ref{t2}), relative to its generators
$g_2,\dots,g_{15}$. The main difference between Proposition \ref{formcom2} and (\ref{t2}) is that in the latter
case, all exponents must be between $0$ and $p-1$.

\begin{prop}\label{formcom3} Relative to $x_1=g_2,\dots,x_{14}=g_{15}$, the relations (\ref{t2}) for $G=G_{15}(p)$
are obtained from Proposition \ref{formcom2} by simply reducing each exponent modulo $p$ (that is,
replacing the given exponent by the remainder of dividing it by $p$), except only when $p=2$ for three
commutators, namely $[g_4,g_2]$, which is trivial, as well as
$$
[g_3,g_2]=g_4*g_5*g_6*g_7*g_{13}*g_{15},\; [g_5,g_3]=g_{11}*g_{15}.
$$
\end{prop}

\begin{proof} If $2\leq j<i\leq 15$, then $[g_i,g_j]$ is the product
of factors $g_r^{a_r}$, $i+j-1\leq r\leq 15$, as indicated in Proposition \ref{formcom2}.
Here $g_r\in K_{r-1}$, so $g_r^{(p)}\in K_{p(r-1)+t}$ by Theorem \ref{kr},
where $t$ is the remainder of dividing $r-1$ by $p$. This implies that $g_r$ has
order $p$ in $G_{15}(p)$, provided $p(r-1)+t\geq 15$. This is automatic if: $r\geq 4$
and $p\geq 5$, which proves the result for these primes; $r\geq 6$ and $p=3$, which proves the result for $p=3$;
$i+j-1\geq 8$ and $p=2$. Thus, it remains to verify the result for $[g_i,g_2]$, $3\leq i\leq 6$, and $[g_i,g_3]$,
$4\leq i\leq 5$. The first factor of $[g_4,g_3]$ is $g_6$, which proves the result in this case.
Now $g_6$, $g_8$, $g_5^{(2)}=g_9$, and $g_7^{(2)}=g_{13}$ are easily seen to have order 2, which proves the result for
$[g_6,g_2]$, $[g_5,g_2]$, $[g_4,g_2]$ (this case is compatible with $g_2^{(2)}=g_4$), $[g_3,g_2]$, and $[g_5,g_3]$.
\end{proof}

It should be clear that the transformation from  the relations in Proposition \ref{formcom2} to those
in Proposition \ref{formcom3} is valid when $R$ is an arbitrary ring of characteristic $p$.

\section{The first half of the power-commutator presentation}\label{pr}

We assume in this section that $R$ has characteristic $p$. The following result appears in \cite[Corollary 6.5]{DS}
culminating a long series of intricate calculations.

\begin{prop}
\label{cc}
Suppose $p\neq 2$ and that $1<i\leq p$. Then
$$
g_i^{(p)}\equiv x+\frac{i}{2}x^{p(i-1)+i}-\frac{i^2}{4}x^{p(i-1)+2i-1}\mod x^{(p+3)(i-1)+1}.
$$
\end{prop}

York \cite{Y} associates an upper triangular infinite matrix $M$ to any $f\in G(R)$, where $M_{i,j}$ is the coefficient of $x^j$
in $f^i$ for all $1\leq i<j$. If the matrix associated to $g\in G$ is $N$, then
the matrix associated to $f*g$ is $MN$. He also defines the matrix $\Delta$ by $M=I+\Delta$, noting that $M^p=I+\Delta^p$,
so that $f^{(p)}$ can be read off from the first row of $\Delta^p$. Given any $d>1$, we let $L(d)$ stand
for the set of
all sequences $j=(j_0,j_1,\dots,j_p)$ of natural numbers such that $1=j_0<j_1<\cdots<j_p=d$, and set $S(j)=\Delta_{j_0,j_1}\cdots \Delta_{j_{p-1},j_p}$. We write $J(d)$ for the possibly empty set of all $j\in L(d)$ such that $S(j)\neq 0$. With this notation, we have
\begin{equation}
\label{d2}
\Delta^p_{1,d}=\underset{j\in J(d)}\sum S(j),
\end{equation}
a formula to be used repeatedly below to compute $f^{(p)}$. If $f=x+x^k$, $k>1$, then for $1\leq i<j$, the binomial expansion gives
$$
\Delta_{i,j}=
\begin{cases} \binom{i}{\ell} & \text{ if }j=i+\ell(k-1),\; 0\leq \ell\leq i,\\
0 & \text{ otherwise}.
\end{cases}
$$
In particular, if $p\mid i$ but $p\nmid j$, then $\Delta_{i,j}=0$, which a special case of a more general result
due to York, as found in \cite[Lemma 3.6]{DS}. We will use the above formula with $k=2$, in which case
\begin{equation}
\label{k2}
\Delta_{i,j}=
\begin{cases} \binom{i}{\ell} & \text{ if }j=i+\ell,\; 0\leq \ell\leq i,\\
0 & \text{ otherwise}.
\end{cases}
\end{equation}

A formula extending Proposition \ref{cc} and giving a closed
formula for the coefficient of $x^{p+4}$ appearing in $g_2^{(p)}$, $p\neq 2$, is not available,
but the following procedure indicates how to obtain this coefficient for any given $p>3$
(the full expansions of $g_2^{(p)}$ when $p=2$ and $p=3$ are easily found), as required below  
when $p\in\{5,7,11\}$.

Set $f=g_2$ and $d=p+4$. Then $J(d)$ consists of all sequences 
$(j_0,j_1,\dots,j_p)$ of natural numbers such that $j_0<j_1<\cdots<j_p$; $j_0=1$, $j_1=2$, and $j_{p+1}=d$;
$j_{i+1}\leq 2 j_i$ for all $0\leq i\leq p$; no $j_i$ is equal to $p$ (as $\Delta_{p,k}=0$ for all $p<k\leq d$; 
if $j_i=p+1$, then $j_{i+1}=p+2$ (since $\Delta_{p+1,p+3}=0=\Delta_{p+1,p+4}$).

Let $J_0(d)$ (resp. $J_1(d)$) be the set of all sequences in $J(d)$ where $p+1$ does not appear (resp. does appear).

Clearly $J_0(d)$ consists of the $p-1$ sequences $j^{p+3},j^{p+2},j^{p-1}, j^{p-2},\dots, j^3$,
where the superscript indicates which number from $\{3,\dots,p-1,p+2,p+3\}$ does not appear in the given sequence,
except that when $p=5$ we have $|J_0(9)|=3$, as $j^4=(1,2,3,7,8,9)\notin J(9)$. To compute $S(j^i)$, we first
note that by Wilson's theorem (to the effect that $(p-1)!\equiv -1\mod p$), we have
\begin{equation}
\label{rec}
S(j^{p+3})=-1,\; S(j^{p+2})=3,\; S(j^{p-1})=-15,\; S(j^{p-2})=-6.
\end{equation}
Moreover, for $4\leq i+1\leq p-2$, we readily verify the following recursive formula:
\begin{equation}
\label{rec2}
S(j^i)=S(j^{i+1})\frac{(i-2)(i+1)}{(i-1)i}.
\end{equation}

We next deal with $J_1(d)$. There are exactly $p-4$ sequences in $J_1(d)$ where 3 is not present, namely
$j^{3,p+3},j^{3,p-1},\dots j^{3,5}$,
where the second superscript indicates which number from $\{5,\dots,p-1,p+3\}$ does not appear in the given sequence
(noting that $4$ must be present). Wilson's theorem yields
\begin{equation}
\label{rec3}
S(j^{3,p+3})=1/6,\; S(j^{3,p-1})=2,\; S(j^{3,p-2})=1.
\end{equation}
Furthermore, for $6\leq i\leq p-2$, we easily derive the following recursive formula:
\begin{equation}
\label{rec4}
S(j^{3,i-1})=S(j^{3,i})\frac{(i-3)i}{(i-1)(i-2)}.
\end{equation}

There are again $p-4$ sequences in $J_1(d)$ where 3 present but 4 is not, namely
$j^{4,p+3},j^{4,p-1},\dots j^{4,5}$,
where the second superscript indicates which number from $\{5,\dots,p-1,p+3\}$ does not appear in the given sequence.
It is easy to obtain the values $S(j^{4,p+3}),\dots, S(j^{4,6})$ (note the omission of $S(j^{4,5})$)
from the corresponding values $S(j^{3,p+3}),\dots, S(j^{3,6})$ found above. For $f_3=\frac{2\times 3}{1\times 4}$, we have
$$
S(j^{4,k})=S(j^{3,k})f_3,\quad k\in\{6,\dots,p-1,p+3\}.
$$

There are: $p-5$ sequences in $J_1(d)$ where $3,4$ are present but 5 is not, namely
$j^{5,p+3},j^{4,p-1},\dots j^{5,6}$; $p-6$ sequences in $J_1(d)$ where $3,4,5$ are present but 6 is not, namely
$j^{6,p+3},j^{6,p-1},\dots j^{6,7}$; and so on. In general, for $5\leq i\leq p-1$, there are $p-i$ sequences
where $3,4,\dots,i-1$ are present but $i$ is not, namely $j^{i,p+3},j^{i,p-1},\dots j^{i,i+1}$
(this list reduces to $j^{p-1,p+3}$ when $i=p-1$). In particular,
$$
|J(d)|=(p-1)+(p-4)+(p-4)+(p-5)+\cdots+1=2p-5+(p-3)(p-4)/2
$$
when $p>5$, and $|J(d)|=5$ when $p=5$. For $3\leq i\leq p-3$, we obtain the $p-(i+2)$ values
$S(j^{i+1,k})$, $k\neq i+2$, from the corresponding values $S(j^{i,k})$, as follows:
\begin{equation}
\label{rec5}
S(j^{i+1,k})=S(j^{i,k})f_i,\quad k\neq i+2,
\end{equation}
where $f_i=\frac{(i-1)i}{(i-2)(i+1)}$. It only remains to compute the $p-4$ special values
$$
S(j^{4,5}),\; S(j^{5,6}),\dots, S(j^{p-3,p-2}),\; S(j^{p-2,p-1}),\; S(j^{p-1,p+3}),
$$
a list that reduces to $S(j^{p-1,p+3})$ when $p=5$. By Wilson's theorem, we have
\begin{equation}
\label{rec6}
S(j^{4,5})=1/10,\; S(j^{p-2,p-1})=15,\; S(j^{p-1,p+3})=2.
\end{equation}
For the remaining values, the following recursive formula is easily derived:
\begin{equation}
\label{rec7}
S(j^{i+1,i+2})=S(j^{i,i+1})\frac{(i-1)i}{(i-3)(i+2)},\quad 5\leq i+1\leq p-3.
\end{equation}

The above procedure and (\ref{d2}) give the coefficient of $x^{p+4}$ in~$g_2^{(p)}$ for all primes 
$p>3$. %It is hardly worth the effort here to attempt similar results for $x^{p+5},x^{p+6},\dots$,
%as these become increasingly more involved and we only need them for $x^{12},x^{13},x^{14},x^{15}$ when $p=7$, in which case a direct approach via (\ref{d2}) yields the desired outcome.

\begin{prop}\label{primera} If $p\geq 17$,
then 
$$g_i^{(p)}=x,\; 2\leq i\leq 15.
$$
If $p=13$, then
$$g_i^{(13)}=x,\; 3\leq i\leq 15,\; g_2^{(13)}=g_{15}.
$$
If $p=11$, then 
$$g_i^{(11)}=x,\; 3\leq i\leq 15,\; g_2^{(11)}=g_{13}*g_{14}^{(10)}*g_{15}^{(7)}.
$$
If $p=7$, then 
$$g_i^{(7)}=x,\; 3\leq i\leq 15,\; g_2^{(7)}=g_9*g_{10}^{(6)}*g_{11}^{(5)}*g_{12}^{(2)}*g_{13}^{(4)}*g_{14}^{(4)}.
$$ 
If $p=5$, then 
$$g_i^{(5)}=x,\; 4\leq i\leq 15,\; g_3^{(5)}=g_{13}^{(4)}*g_{15}^{(4)},\;
g_2^{(5)}=g_7*g_8^{(4)}*g_9^{(4)}*g_{10}^{(4)}*g_{11}^{(2)}*g_{13}^{(4)}*g_{14}^{(2)}.
$$
If $p=3$, then 
$$g_i^{(3)}=x,\; 5\leq i\leq 15, g_4^{(3)}=g_{10},\; g_3^3=x,\;
g_2^{(3)}=g_5*g_6^{(2)}*g_7*g_8*g_{10}^{(2)}*g_{11}*g_{12}*g_{14}.
$$
If $p=2$, then $g_i^{(2)}=x$ for all $9\leq i\leq 15$,  and
$$
g_2^{(2)}=g_4,\; g_4^{(2)}=g_6^{(2)}=g_8^{(2)}=x,\; g_3^{(2)}=g_5*g_7*g_9*g_{11}*g_{13},\; g_5^{(2)}=g_9,\;
g_7^{(2)}=g_{13}.
$$
\end{prop}

\begin{proof} If $p>13$, Proposition \ref{kr} applies. If $p=13$, we apply Proposition \ref{kr}
when $i>2$ and Proposition \ref{cc} when  $i=2$. If $p=11$, then Proposition \ref{kr} yields the result when $i>2$, and we deal with the case $i=2$ below. If $p=7$, we apply Proposition \ref{kr}
when $i>2$, and deal with the case $i=2$ below. If $p=5$, we apply Proposition \ref{kr}
when $i>3$, use Proposition \ref{cc} and Lemma \ref{rad} when $i=3$, and deal with the case $i=2$ below.
If $p=3$, we apply Proposition \ref{kr}
when $i>4$; from direct computation $g_4^{(2)}=x+2x^4+x^7+x^{13}$ and $g_4^{(3)}=g_{10}$; 
again, from direct computation $g_3^{(2)}=x+2x^3+x^9$ and $g_3^{(3)}=x$ (a more general
argument appears in \cite[Example 6.8]{DS}), 
and deal with the case $i=2$ below. If $p=2$, Theorem \ref{kr} gives $g_i^{(2)}=x$ for all $i>8$; moreover,
direct computation yields $g_4^{(2)}=g_6^{(2)}=g_8^{(2)}$ as well as $g_7^{(2)}=g_{13}$,
$g_5^{(2)}=g_9$, $g_3^{(2)}=x+x^5+x^7+x^9=g_5*g_7*g_9*g_{11}*g_{13}$, and $g_2^{(2)}=g_4$.

It remains to write $g_2^{(p)}$ when $p\in\{3,5,7,11\}$ in terms the generators indicated in (\ref{t1}).
By Proposition~\ref{cc}, we have
$$
g_2^{(3)}\equiv x+x^5-x^6\mod x^7,\; g_2^{(5)}\equiv x+x^7-x^8\mod x^9,
$$
$$
g_2^{(7)}\equiv x+x^9-x^{10}\mod x^{11},\; g_2^{(11)}\equiv x+x^{13}-x^{14}\mod x^{15},
$$
and we proceed to determine the remaining coefficients.

$\bullet$ $p=3$. We have $g_2^{(2)}=x+2x^2+2x^3+x^4=w$, so $g_2^{(3)}=w+w^2$ has degree 8, where
the coefficients of $x^7$ and $x^8$ are easily seen to be equal to 1, so 
$$g_2^{(3)}=x+x^5-x^6+x^7+x^8=x+d_5x^5+\cdots+d_{15}x^{15}.$$
The argument used to compute $[g_2,g_3]$ at the end of the proof of Proposition \ref{formcom2}
applies, yielding 
$$
c_6=-1,\; c_7=1,\; c_8=1,\; c_9=0,\; c_{10}=-1,
c_{11}=1,\; c_{12}=1,\;c_{13}=0,\;c_{14}=-1,\; c_{15}=0,
$$
$$
a_6=-1,\; a_7=1,\; a_8=1,\; c_9=0,\; a_{10}=-1,
a_{11}=1,\; a_{12}=1,\;a_{13}=0,\;a_{14}=1,\; a_{15}=0,
$$
so that
$$
g_2^{(3)}=g_5*g_6^2*g_7*g_8*g_{10}^2*g_{11}*g_{12}*g_{14}.
$$

$\bullet$ $p=5$. We wish to compute the coefficient of $x^d$ in $g_2^{(5)}$ for  $9\leq d\leq 15$.
As indicated above, this coefficient is equal to $\Delta^5_{1,d}$ and can be computed by means of 
(\ref{d2}) and (\ref{k2}).  For each $9\leq d\leq 15$, we list all sequences $j\in J(d)$
followed by the corresponding values $S(j)$, and their sum $\Delta^5_{1,d}$. If $j\in J(9)$, then $j$ is one of
$$
(1,2,3,4,7,9), (1,2,3,4,8,9), (1,2,3,6,7,9), (1,2,4,6,7,9), (1,2,4,7,8,9),
$$
which corresponding values $S(j)$ adding to $(-1)+(-2)+2+1+(-1)=-1$. Note that these values $S(j)$ agree with the
ones indicated in the above procedure. If $j\in J(10)$, then $j$ is one of
$$
(1,2,3,4,5,10), (1,2,3,4,8,10), (1,2,4,7,8,10), (1,2,4,7,9,10), (1,2,4,8,9,10),
$$
which corresponding values $S(j)$ adding to $(-1)+(-2)+(-1)+1+2=-1$. If $j\in J(11)$, then $j$ is one of
$$
(1,2,3,4,6,11),  (1,2,3,4,8,11), (1,2,4,7,8,11), (1,2,4,7,9,11), (1,2,4,8,9,11),
$$
which corresponding values $S(j)$ adding to $1+1+(-2)+(-1)+(-2)=2$. If $j\in J(12)$, then $j$ is one of
$$
(1,2,3,4,6,12), (1,2,3,4,7,12), (1,2,3,6,7,12), (1,2,3,6,11,12), (1,2,4,6,7,12), 
$$
$$
(1,2,4,6,11,12), (1,2,4,7,9,12),
(1,2,4,8,9,12), (1,2,4,8,11,12),
$$
which corresponding values $S(j)$ adding to $1+(-1)+2+2+1+1+1+2+1=0$. If $j\in J(13)$, then $j$ is one of
$$
(1,2,3,4,7,13), (1,2,3,4,8,13), (1,2,3,6,7,13), (1,2,3,6,12,13), (1,2,4,6,7,13),
$$
$$
(1,2,4,6,12,13), (1,2,4,7,8,13), (1,2,4,7,9,13), (1,2,4,7,12,13),  (1,2,4,8,9,13), 
$$
whose values $S(j)$ add to $(-2)+1+(-1)+(-1)+2+2+(-2)+(-1)+(-2)+(-2)=-1$. If $j\in J(14)$, then $j$ is one of
$$
(1,2,3,4,7,14), (1,2,3,4,8,14), (1,2,3,6,7,14), (1,2,3,6,12,14), (1,2,4,6,7,14),(1,2,4,6,12,14), 
$$
$$
(1,2,4,7,8,14), (1,2,4,7,9,14), (1,2,4,7,12,14), (1,2,4,7,13,14), (1,2,4,8,9,14), (1,2,4,8,13,14),
$$
whose values $S(j)$ add to $(-1)+(-2)+2+2+1+1+(-1)+(-1)+(-1)+(-1)+(-2)+(-2)=0$. If $j\in J(15)$, then $j$ is one of
$$
(1,2,3,5,10,15),  (1,2,4,5,10,15), (1,2,3,4,8,15), (1,2,4,7,8,15), (1,2,4,7,9,15), (1,2,4,7,13,15),
$$
$$
(1,2,4,7,14,15),(1,2,4,8,9,15),(1,2,4,8,10,15),(1,2,4,8,13,15), (1,2,4,8,14,15),
$$
whose values $S(j)$ adding to $2+(-2)+(-2)+(-1)+1+(-2)+1+2+1+1+2=1$.Thus 
$$g_2^{(5)}=x+x^7-x^8-x^9-x^{10}+2x^{11}-x^{13}+x^{15},
$$ 
so Proposition \ref{rad8} yields 
$g_2^{(5)}=g_7*g_8^{(4)}*g_9^{(4)}*g_{10}^{(4)}*g_{11}^{(2)}*g_{13}^{(4)}*g_{14}^{(2)}$.

$\bullet$ $p=11$. It remains to compute the coefficient of $x^{15}$ in $g_2^{(11)}$, and we apply the above procedure
for this purpose. For the sequences in $J_0(11)$, besides the explicit 4 values given in (\ref{rec}), we appeal to~(\ref{rec2})
to derive
$$
S(j^8)=-5,\; S(j^7)=1,\; S(j^6)=-2,\; S(j^5)=-4,\; S(j^4)=4,\; S(j^3)=-1.
$$
Regarding $J_1(11)$, besides the 3 values given in (\ref{rec3}) for the sequences where 3 does not appear,
we resort to (\ref{rec4}) to obtain
$$
S(j^{3,8})=-1,\; S(j^{3,7})=-2,\; S(j^{3,6})=4,\; S(j^{3,5})=-3.
$$
Appealing to (\ref{rec5})-(\ref{rec7}) we further derive
$$
S(j^{4,14})=3,\; S(j^{4,10})=3,\; S(j^{4,9})=-4,\; S(j^{4,8})=4,\; S(j^{4,7})=-3,\; S(j^{4,6})=-5,\; S(j^{4,5})=-1,
$$
$$
S(j^{5,14})=-3,\; S(j^{5,10})=-3,\; S(j^{5,9})=4,\; S(j^{5,8})=-4,\; S(j^{5,7})=3,\; S(j^{5,6})=-2,
$$
$$
S(j^{6,14})=4,\; S(j^{6,10})=4,\; S(j^{6,9})=2,\; S(j^{6,8})=-2,\; S(j^{6,7})=5,
$$
$$
S(j^{7,14})=-2,\; S(j^{7,10})=-2,\; S(j^{7,9})=-1,\; S(j^{7,8})=-2,
$$
$$
S(j^{8,14})=-1,\; S(j^{8,10})=-1,\; S(j^{8,9})=5,\; S(j^{9,14})=1,\; S(j^{9,10})=4,\; S(j^{10,14})=2.
$$
Adding the foregoing 45 values $S(j)$, we deduce from (\ref{d2}) that the sought coefficient is equal to 7.
Thus, $g_2^{(11)}=x+x^{13}-x^{14}+7x^{15}$, hence $g_2^{(11)}=g_{13}*g_{14}^{(10)}*g_{15}^{(7)}$ by Proposition \ref{rad}.

$\bullet$ $p=7$. We first compute $\Delta^7_{1,11}$ by means of the above procedure.
For the sequences in $J_0(11)$, besides the explicit 4 values given in (\ref{rec}), we apply (\ref{rec2})
to derive
$$
S(j^4)=2,\; S(j^3)=-1.
$$
Regarding $J_1(11)$, the 3 values given in (\ref{rec3}) are complete for the sequences where 3 does not appear.
Appealing to (\ref{rec5})-(\ref{rec7}) we  further obtain
$$
S(j^{4,10})=2,\; S(j^{4,6})=3,\; S(j^{4,5})=5,\; S(j^{5,10})=1,\; S(j^{5,6})=1,\; S(j^{6,10})=2.
$$
Adding the foregoing 15 values $S(j)$, we deduce from (\ref{d2}) that $\Delta^7_{1,11}=5$.

We next use the above information to produce sufficient data to obtain $\Delta^7_{1,12}$,
and repeat this step until all of $\Delta^7_{1,12}, \Delta^7_{1,13}, \Delta^7_{1,14},\Delta^7_{1,15}$
are found.

Let $d>1$. To any sequence $j\in J(d)$ we associate the sequence $j'$ obtained
by replacing $j_7=d$ by $j'_7=d+1$. The set of all $j\in J(d)$ such that $j'\in J(d+1)$
will be denoted by $K(d)$. To any 
sequence $j\in J(d)$ such $j_{5}=d-2$, $j_6=d-1$, we associate the sequence $j^\sharp$ obtained
by replacing $j_6,j_7$ by $j^\sharp=d, j^\sharp=d+1$, respectively.

Of the 15 sequences in $J(11)$, exactly 5 of them end with $9,11$, and the remaining 10 end with $10,11$
and form $K(10)$. Clearly $S(j')=S(j)$ for every $j\in K(10)$ (this and subsequent statements of
this kind follow directly from the way the new sequence is formed from the old one). 
Every remaining sequence in $J(12)$ must end in $11,12$. For
each sequence $j\in J(11)$ ending in $9,10,11$, and there 9 of these, the sequence $j^\sharp$,
ending in $9,11,12$, is in $J(12)$, and every sequence in $J(12)$ ending in $9,11,12$ originates
in this way. Moreover, we have $S(j^\sharp)=3 S(j)$. There is a single sequence in $J(11)$ ending in $6,10,11$
and a single sequence in $J(12)$ ending in $6,11,12$, and the $S$-value of these
sequences is identical. Every remaining 
sequence in $J(12)$ must end in $10,11,12$. There are 9 of them, 6 beginning with $1,2,3$ and the 3
starting with $1,2,4$. The presence of these common beginning and endings terms makes the calculation of the $S$-values 
of these sequences, and their sum, particularly easy. Adding the $S$-values of all 10+9+1+6+3=29 sequences described above yields 
$\Delta^7_{1,12}=2$.

Every one of the 29 sequences $j$ in $J(12)$ produces a sequence $j'$ in $J(13)$,
and $S(j')=-2 S(j)$. Every remaining sequence in $J(12)$ must end in $12,13$.
For each sequence $j\in J(12)$ ending in $10,11,12$, and there 9 of these, 
the sequence $j^\sharp$,
ending in $10,12,13$, is in $J(13)$, and every sequence in $J(13)$ ending in $10,12,13$ originates
in this way. Moreover, we have $S(j^\sharp)=3 S(j)$. There is a single sequence in $J(12)$ ending in $6,11,12$
and a single sequence in $J(13)$ ending in $6,12,13$, and the $S$-value of these
sequences is identical. Every remaining 
sequence in $J(13)$ must end in $11,12,13$. There are 12 of them, 6 beginning  with $1,2,3$ and 6
starting with $1,2,4$. The presence of these common beginning and endings terms makes the calculation of the $S$-values 
of these sequences, and their sum, particularly easy. Adding the values of all 29+9+1+6+6=51 sequences described above yields $\Delta^7_{1,13}=4$.

Of the 51 sequences in $J(13)$, exactly 10 of them end with $10,13$, and the remaining 41 end with $11,13$ or $12,13$
and form $K(13)$. Clearly $S(j')=2S(j)$ in the first case and $S(j')=3S(j)$ in the
second. Every remaining sequence in $J(14)$ must end in $13,14$. For
each sequence $j\in J(13)$ ending in $11,12,13$, and there 12 of these, the sequence $j^\sharp$, ending in $11,13,14$,
in in $J(14)$. Clearly every sequence in $J(14)$ ending in $11,13,14$ originates
in this way, and we have $S(j^\sharp)=-S(j)$. Except for $(1,2,3,4,5,6,7,14)$, which has $S$-value $-1$ by Wilson's theorem,
every remaining 
sequence in $J(14)$ must end in $12,13,14$ or 10,13,14. There are 9 sequence of each type, and
in both cases 5 of them start with $1,2,3$, and 4 of them with $1,2,4$. 
The presence of these common beginning and endings terms makes the calculation of the $S$-values 
of these sequences, and their sum, particularly easy. Adding the values of all 41+12+1+5+4+5+4=72 sequences described above yields $\Delta^7_{1,14}=4$.

From the 72 sequences in $J(14)$ and their values $S$-values as described above, we 
immediately obtain the 72 sequences in $J(15)$ and their values $S$-values, as follows. Except for $(1,2,3,4,5,6,7,14)$,
every other $j\in J(14)$ has the property that $j'\in J(15)$. The only sequence in $J(15)$
which does not originate in this manner is $(1,2,3,4,5,6,8,15)$.
Thus $|J(15)|=72$. Let $A,B,C$ be the sum of the values $S(j)$ such that $j\in J(14)$ and $j_6$ is equal to $13,12,11$, respectively
(these are the only 3 possibilities for $j_6$). Moreover, let $A',B',C'$ be the sum of all $S(j')$ 
such that $j_6$ is equal to $13,12,11$, respectively. Then $A'=-A, B'=B, C'=2C$. Since $S(1,2,3,4,5,6,8,15)=1$ by Wilson's theorem,
it follows that $\Delta^7_{1,15}=-A+B+2C+1$. Here $A=3,B=2,C=0$ by the previous case, so $\Delta^7_{1,15}=0$.

Thus, $g_2^{(7)}=x+x^{9}-x^{10}+5x^{11}+2x^{12}+4x^{13}+4x^{14}$, so
$g_2^{(7)}=g_{9}*g_{10}^{(6)}*g_{11}^{(5)}*g_{12}^{(2)}*g_{13}^{(4)}*g_{14}^{(4)}$ by Proposition \ref{rad}.
\end{proof}

\begin{theorem}\label{pcp} The power-commutator presentation (\ref{t1})-(\ref{t2}) of $G_{15}(p)$ relative to its generators 
$x_1=g_2,\dots,x_{14}=g_{15}$ is as outlined in Proposition \ref{formcom3} (referring to Proposition \ref{formcom2}) 
regarding~(\ref{t2}), and as indicated in Proposition \ref{primera} in regards to (\ref{t1}).
\end{theorem}

\noindent{\bf Acknowledgment.} I am very grateful to Darnel Cantor Fierro and Allen Herman for the GAP verification
of the PCP of $G_{15}(p)$.

%Direct computation in $G_{15}(\Z)$ gives $g_2^{(2)}=x+2x^2+2x^3+x^4$, which implies
%\begin{equation}
%\label{ptres}
%g_2^{(3)}=x+3x^2+6x^3+9x^4+10x^5+8x^6+4x^7+x^8.
%\end{equation}

%===================================================

\end{document}